\documentclass[12pt]{article}
\usepackage{amsmath,amssymb,amsthm}

\numberwithin{equation}{section}

\newtheorem{thm}{Theorem}

\newcommand{\n}{\nonumber}

\renewcommand{\t}{\theta}

\newcommand{\np}{\nabla ^\perp}

\newcommand{\gl}{{\gamma}}

\newcommand{\cc}{\mathcal{ D}}

\newcommand{\ts}{T}
\renewcommand{\tt}{(T-t)}

\renewcommand{\o}{\omega}

\renewcommand{\O}{\Omega}

\newcommand{\bb}{\begin{equation}}
\newcommand{\ee}{\end{equation}}
\newcommand{\bq}{\begin{eqnarray}}
\newcommand{\eq}{\end{eqnarray}}
\newcommand{\bqn}{\begin{eqnarray*}}
\newcommand{\eqn}{\end{eqnarray*}}
\newcommand{\pt}{\partial}
\newcommand{\de}{\delta}

\begin{document}
\title{Remark on Luo-Hou's ansatz for a self-similar solution to the 3D Euler equations}
\author{Dongho Chae$^*$  and Tai-Peng Tsai$^\dagger$\\
\ \\
 $*$Department of Mathematics\\
Chung-Ang University\\
 Seoul 156-756, Republic of Korea\\
email: dchae@cau.ac.kr\\
and \\
$\dagger$Department of Mathematics\\
 University of British Columbia\\
    Vancouver, B.C. V6T 1Z2, Canada\\
email: ttsai@math.ubc.ca\\
and\\
$\dagger$Center for Advanced Study in Theoretical Sciences\\
National Taiwan University\\
Taipei, Taiwan}
\date{}
\maketitle

\begin{abstract}

In this note we show that Luo-Hou's ansatz for the self-similar solution to the axisymmetric solution to the 3D Euler equations
leads to triviality of the solution under suitable decay condition of the blow-up profile. The equations for the blow-up profile reduces to an over-determined system of partial differential equations, whose only solution with decay is the trivial solution.
We also propose a generalization of Luo-Hou's ansatz.  Using the vanishing of the normal velocity at the boundary, we show that this  generalized  self-similar ansatz also leads to a trivial solution.
These results show that the self-similar ansatz may be valid either only in a time-dependent region which shrinks to the boundary circle at the self-similar rate, or under different boundary conditions at spatial infinity of the self-similar profile.
  \noindent{\bf
keywords:} Euler equations, finite time blow-up, Luo-Hou's ansatz, self-similar solution
\end{abstract}

\section{Axisymmetric 3D Euler system}
\setcounter{equation}{0}

We are concerned with the homogeneous incompressible 3D Euler equations,
$$
(E) \left\{\aligned  &\partial_t  u+u\cdot \nabla u =-\nabla p,\label{e1}\\
&\mathrm{div }\, u=0,
\endaligned \right.
$$
 where $u(x,y,z,t)$ is the velocity vector field, and $p=p(x,y,z,t)$ is the scalar
  pressure.
  We consider an axisymmetric solution of the Euler equations,
which means that the velocity field $u$ has the representation
$$
u=u^r (r,z,t)e_r +u^\theta (r,z,t)e_\theta +u^z
(r,z,t)e_z
$$
 in the cylindrical coordinate system $(r,\theta,z)$, where
 $$e_r = \left(\frac{x}{r}, \frac{y}{r}, 0\right), \quad
 e_\theta = \left(-\frac{y}{r}, \frac{x}{r}, 0\right),\quad e_z=(0,0,1),\quad
 r=\sqrt{x^2 +y^2}.
 $$
  Let $\psi$ be the vector stream function satisfying, curl $\psi=u$ and  div $\psi=0$, and $\psi^\theta$ be its angular component.
 Let $\omega=\textrm{curl } u$ be the vorticity vector and $\omega^\theta$ be its angular component.
 Then, the Euler equations for the axisymmetric
solution can be written as (see \cite{maj})
\bq
\label{ae1}
 &&\partial_t u^\theta +u^r\partial_r u^\theta  +u^z\partial_z u^\theta =-\frac{u^r u^\t}{ r} ,\\
  \label{ae2}
&&\partial_t \o^\theta  +u^r\partial_r \o^\theta  +u^z\partial_z \o^\theta =
\frac{2u^\theta}{r} \partial_z u^\theta+\frac{1}{r} u^r\o^\theta\\
 \label{ae3}
&&-\left(\Delta -\frac{1}{r^2}\right)\psi^\theta =\o^\theta.
\eq
 In order to remove the artificial
singularity at $r=0$ of the original system we introduce $(u_1, \o_1, \psi_1)$ defined by
  \bb\label{def}
u_1=\frac{u^\theta}{r},\quad \o_1=\frac{\o^\theta}{r}, \quad  \psi_1=\frac{\psi^\theta}{r}.
  \ee
 Then, the system (\ref{ae1})--(\ref{ae3})
  can be written  in terms of  $( u_1,\o_1,\psi_1)$ as
 \bq\label{01}
 &&\partial_t u_1 +u^r \partial_r  u_1 +u^z \partial_z u_1=2u_1 \partial_z \psi_1,\\
 \label{02}
 &&\partial_t \o_1 +u^r \partial_r  \o_1 +u^z \partial_z \o_1= \partial_z (u_1^2 ),\\
 \label{03}
 &&-\left( \partial_r ^2 +\frac{3}{r} \partial_r +\partial_z^2 \right)\psi_1=\o_1,
 \eq
 where
 \bb\label{conv}  u^r=-r\partial_z \psi_1,\quad u^z=2\psi_1+ r\partial_r \psi_1.
 \ee

\section{Lou-Hou's self-similar ansatz}
\setcounter{equation}{0}

 We consider the system (\ref{01})--(\ref{03}) in the infinite cylinder
 $$\{ (r,z) \in \Bbb R^2\, |\, 0< r< 1, -\infty<z<\infty\}$$
on $0\leq t<T$, where $T$ is a possible blow-up time.  For a possible blow-up scenario at the
 circle on the boundary of the cylinder, observed numerically in \cite{hou}, Luo-Hou  \cite[\S4.7]{hou} proposed the following self-similar ansatz  for the solutions to (\ref{01})--(\ref{03}), 
 \bq\label{11}
 u_1(r,z, t)&=&(\ts-t)^{-1+\frac{\gl}{2}} U\left(R,Z \right),\\
 \label{12}
 \o_1(r,z,t)&=&(\ts-t)^{-1} \O\left(R,Z\right),\\
 \label{13}
   \psi_1(r,z,t)&=&(\ts-t)^{-1+2\gl} \Psi\left(R,Z\right),
   \eq
 where
\bb\label{13a}
 R=\frac{r-1}{(\ts-t)^\gl}, \qquad Z=\frac{z}{(\ts-t)^\gl},
\ee
 and  $\gl\geq 2/5$,  which is valid on a neighborhood of the circle on the boundary for all time sufficiently close to the possible blow-up time. The region  $D_\infty(t)$ of self-similarity studied in \cite{hou} is defined dynamically as where the vorticity magnitude exceeds one half of its maximal magnitude at each time $t$. 
 It is observed
numerically to shrink to the boundary circle as $t\to T_-$. If  (\ref{11})--(\ref{13}) are valid
in the set $D_\infty(t)$, the diameter of $D_\infty(t)$ should be proportional to $(T-t)^{\gamma}$ and corresponds to a fixed finite set in the left half $RZ$-plane. However the self-similar ansatz (\ref{11})--(\ref{13}) could be valid in a larger space-time set.

For our analysis below, we will assume that the self-similar ansatz \eqref{11}--\eqref{13} is valid either in the space-time region
 \bb \label{thm2region}
\mathcal{ C}_{\delta,T} :=\{ (r,z,t)\in \Bbb R^3\, |\, 1-\de<r<1, \quad -\de <z<\de , \quad T-\de<t<T\},
\ee
for some $0<\delta \ll 1$,
or in the region
 \bb \label{thm2region2}
\mathcal{ W}_{\delta(t)} :=\{ (r,z,t)\in \mathbb{ R}^3\, |\, 
1-\de(t)< r< 1, \quad -\de(t)<z<\de(t) , \quad T_0<t<T\}
\ee
where $\de(t)>0$ is a decreasing function of $t\in (T_0,T)$ for some $T_0<T$
and 
\bb \label{thm2region2b}
\lim _{t \to T_-}\de(t)=0,
\quad \limsup _{t \to T_-}(T-t)^{-\gamma}\de(t)=\infty.
\ee
   Note that,  in either case, $(U, \O, \Psi)$ is defined on the left half-plane,
\bb
\cc=\{Y=(R, Z )\in \Bbb R^2\, |\, -\infty<R \leq 0, -\infty<Z<\infty \}.
\ee
We will verify that the above ansatz reduces to the triviality for the solution to (\ref{01})--(\ref{03}).

\begin{thm}\label{thm1}
Let $(u_1, \o_1, \psi_1)$ be  a classical solution  to the system (\ref{01})--(\ref{03}) 
with the representation
\eqref{11}--\eqref{13}, $0<\gamma<\infty$, in either the set $\mathcal{ C}_{\delta,T}$
defined by \eqref{thm2region}, or in the set $\mathcal{ W}_{\delta(t)}$ defined by 
\eqref{thm2region2}--\eqref{thm2region2b}.
We assume the following asymptotic condition for the blow-up profiles $(U, \O)$,
\bb
\label{26}
|U(Y)|+|\O(Y)|=o(1) \quad\mbox{as $|Z|\to \infty$}.
\ee
Then,  $u_1=\o_1=0$, and $\psi_1=\psi_1(z,t)=a\tt^{-1+\gl}z+b\tt^{-1+2\gl}$ for some constants $a, b$.
\end{thm}
\noindent{\em Remark}. In Section 4 we  partially explain why the condition (\ref{26}) is a natural decay condition for the blow-up profiles. Note that we do not assume any decay in $R$. There is no boundary condition for (\ref{01})--(\ref{03}) at $r=1$.

\begin{proof} [Proof  of Theorem 1]
We first observe from (\ref{conv}) that
  \bq\label{14}
  u^r&=& -\{ \tt ^\gl R +1\}(\ts -t)^{-1+\gl} \partial_Z \Psi, \\
  \label{15}
   u^z&=& 2(\ts -t) ^{-1+2\gl}\Psi +\{ \tt ^\gl R +1\}  (\ts-t)^{-1+\gl} \partial_R \Psi.
   \eq
   Substituting (\ref{11})--(\ref{13}) into (\ref{01})--(\ref{03}), one obtains
   \bq\label{16}
   &&\left(1-\frac{\gl}{2}\right)(\ts-t)^{-2+\frac{\gl}{2}} U+\gl \tt^{-2+\frac{\gl}{2}} (R\partial_R +Z\partial_Z ) U\n \\
    &&\quad-\{ \tt ^\gl R +1\}\tt ^{-2+\frac{\gl}{2}} \partial_Z \Psi \partial_R U\n \\
   &&\qquad+\left[ 2\tt ^{-1+2\gl}\Psi +\{ \tt ^\gl R +1\} \tt ^{-1+\gl} \partial_R \Psi \right]\tt ^{-1-\frac{\gl}{2}}\partial_Z U\n\\
   &&\quad=2\tt ^{-2+\frac{3}{2}\gl}U\partial_Z \Psi,
   \eq
   \bq\label{17}
   && \tt^{-2}\O+\gl \tt ^{-2} (R\partial_R +Z\partial_Z )\O\n \\
   &&\quad-
   \{ \tt ^\gl R +1\}\tt^{-2} \partial_Z \Psi \partial_R \O\n\\
   &&\qquad+\left[ 2\tt ^{-1+2\gl} \Psi +\{ \tt ^\gl R +1\}\tt ^{-1+\gl}\partial_R \Psi \right] \tt ^{-1-\gl} \partial_Z \O\n\\
   &&\quad= \tt^{-2} \partial_Z U^2,
   \eq
   and
   \bq\label{18}
  - \tt^{-1}(\partial_R^2 \Psi + \partial_Z ^2 \Psi )
   -\frac{3\tt^{-1+\gl}}{\{ \tt ^\gl R +1\}}  \partial_R \Psi
    =\tt ^{-1}\O.
   \eq
   The equations (\ref{16})--(\ref{18}) are valid for all $t$ sufficiently close to $\ts$.
   We obtain from (\ref{16})--(\ref{18}) the equations for the most dominant terms as $t\nearrow \ts$,
   \bq\label{19}
   &&\left(1-\frac{\gl}{2}\right) U+\gl Y\cdot \nabla U +\np\Psi \cdot \nabla U=0,\\
   \label{20}
  &&\O+\gl Y\cdot \nabla\O +\np\Psi \cdot \nabla \O=\partial_Z U^2,\\
  \label{21}
  &&\qquad-\Delta\Psi =\O,
   \eq
   where we  denoted
 $$\nabla = (\partial_R, \partial_Z ),\quad \np=(-\partial_Z, \partial_R), \quad
 \Delta =\partial_R^2 +\partial_Z ^2 .$$
  The next dominant equations from  (\ref{16})--(\ref{18}) as $t \nearrow \ts$ are
   \bq\label{22}
   &&R \np \Psi \cdot \nabla U+2\Psi \partial_Z U=2U\partial_Z \Psi,\\
   \label{23}
   && R \np \Psi \cdot \nabla \O+2\Psi \partial_Z \O=0,\\
   \label{24}
   && \qquad\partial_R \Psi=0.
   \eq
 From (\ref{24}) we have $\Psi=\Psi(Z)$ on $\cc$. From this and (\ref{21}) we also have $\O=\O(Z)$.
 Therefore, from (\ref{20}) we have $ U^2(Y)= f(Z)+g(R)$ for some functions $f, g$. Since $U^2$ vanishes as $|Z|\to \infty$,
 $g=$ constant independent of $R$, and we have $U=U(Z)$. Thus, (\ref{19}) reduces to
 $$
 \left(1-\frac{\gl}{2}\right) U+\gl Z \partial_Z U=0.
 $$
 If $\gl\neq 2$, then the maximum principle together with the condition $|U|=o(1)$ as $|Z|\to \infty$ implies $U=0$.
 If $\gl=2$, then from $2Z\partial_Z U=0$ we deduce $U(Z)=$ constant$=0$ for $Z\neq0$. By continuity $U|_{Z=0}=0$ also.
 Substituting $U=0$, $\O=\O(Z)$ and $\Psi =\Psi (Z)$ into (\ref{20}), we find
 $$\O+ \gl Z\partial_Z \O=0
 $$
with  $\gl >0$. The maximum principle together with the condition $|\O|=o(1)$ as $|Z|\to \infty$ implies $\O=0$.
 From (\ref{21}) we find that the function  $\Psi$ satisfies
 $\Psi^{\prime\prime}(Z)=0$, and we have $\Psi (Z)=aZ +b$  on $\cc$.
\end{proof}
\section{Generalized self-similar ansatz}
\setcounter{equation}{0}

Unlike the usual self-similar ansatz for a singularity at the origin, the terms in \eqref{16}--\eqref{18} do not have equal factors of powers of $T-t$. Indeed, they differ by integer powers of $(T-t)^\gl$.
Thus it seems natural to add higher order terms to  Luo-Hou's ansatz and propose the following
\bq\label{11k}
 u_1(r,z, t)&=&(\ts-t)^{-1+\frac{\gl}{2}} \sum_{k=0}^\infty \tt^{k\gl}U_k(R,Z),\\
 \label{12k}
 \o_1(r,z,t)&=&(\ts-t)^{-1}\sum_{k=0}^\infty \tt^{k\gl}\Omega_k(R,Z),\\
 \label{13k}
   \psi_1(r,z,t)&=&(\ts-t)^{-1+2\gl} \sum_{k=0}^\infty \tt^{k\gl}\Psi_k(R,Z).
   \eq
This ansatz contains \eqref{11}--\eqref{13} as a special case by setting $U_k=\O_k=\Psi_k=0$ for $k>0$. The equations for the most dominant terms as $t\to T$ are the same as \eqref{18}--\eqref{20} with $U,\O,\Psi$ replaced by
 $U_0,\O_0,\Psi_0$, see  \eqref{28}--\eqref{30} below. The equations for the next dominant terms are however different:
\bq
\nonumber
  &&
\left(1-\frac{3\gl}{2}\right) U_1+\gl Y\cdot \nabla U_1 +\np\Psi _0\cdot \nabla U_1
+\np\Psi _1\cdot \nabla U_0 \\
\label{3.4}
&&\qquad
+R \np \Psi_0 \cdot \nabla U_0+2\Psi _0\partial_Z U_0=2U_0\partial_Z \Psi_0,\\
 \nonumber
   &&
(1-\gl)\O_1+\gl Y\cdot \nabla\O_1 +\np\Psi_0 \cdot \nabla \O_1+\np\Psi_1 \cdot \nabla \O_0\\
&&\qquad
 +R \np \Psi_0 \cdot \nabla \O_0+2\Psi_0 \partial_Z \O_0=\partial_Z ( 2U_0 U_1),\\
   \label{3.6}
   && -\Delta\Psi_1+ \partial_R \Psi_0=\O_1.
   \eq
Our argument in the previous section does not work for such an ansatz.

However, we will show that such generalized ansatz still has no nontrivial solution
using the following observation on the boundary condition.
In Section 2 we did not assume any boundary condition on the $Z$-axis. However, since
$u^r=-r \partial_z \psi_1$ has the natural boundary condition $u^r=0$ at $r=1$, it is natural to assume
\bb
\label{Psi-assume}
\pt_Z \Psi|_{R=0}=0.
\ee
With a similar assumption on $\Psi_k$, Theorem \ref{thm2} below asserts the triviality of the ansatz \eqref{11k}--\eqref{13k}, which gives an alternative proof of  Theorem \ref{thm1} if we also assume decay in $R$ in \eqref{26}.

\begin{thm}\label{thm2}
Let $(u_1, \o_1, \psi_1)$ be  a classical solution  to the system (\ref{01})--(\ref{03}) 
with the representation
\eqref{11k}--\eqref{13k} for some $0<\gamma<\infty$, in either the set $\mathcal{ C}_{\delta,T}$
defined by \eqref{thm2region}, or in the set $\mathcal{ W}_{\delta(t)}$ defined by 
\eqref{thm2region2}--\eqref{thm2region2b}.
We assume the following  conditions:
\bb
U_k,\O_k,\Psi_k \in C^1_{loc}(\overline{\mathcal{D}}), \quad  \forall k \ge 0,
\ee
\bb
\label{UOPsi-infty}
|U_k(Y)|+|\O_k(Y)|=o(1) ,\quad  |\nabla \Psi_k(Y)|=o(|Y|) \quad\mbox{as $|Y|\to \infty$},\quad  \forall k \ge 0,
\ee
\bb
\label{Psi-BC}
\pt_Z \Psi_k|_{R=0}=0,\quad  \forall k \ge 0,
\ee
and, for some even integer $p$,
\bb
\label{integral-assume}
\lim_{\rho \to \infty}
\int_{\rho<|Y|<2\rho} (U_k^{p} + \O_k^{p}) dY= 0 ,\quad \forall k \le  1/ \gamma.
\ee
Then  $u_1=\o_1=0$ and $\nabla \psi_1=0$.
\end{thm}

\begin{proof}
We will show that  $U_k= \O_k=0$ and $\nabla \Psi_k=0$ for $k\ge 0$ by induction.

We first observe that, as in Section 2, the equation for the most dominant terms are
 \eqref{16}--\eqref{18} with  $U$, $\Omega$ and $\Psi$ replaced by $U_0$, $\Omega_0$ and $\Psi_0$,
  \bq\label{28}
   &&\left(1-\frac{\gl}{2}\right) U_0+\gl Y\cdot \nabla U_0 +\np\Psi_0 \cdot \nabla U_0=0,\\
   \label{29}
  &&\O_0+\gl Y\cdot \nabla\O_0 +\np\Psi_0 \cdot \nabla \O_0=\partial_Z U_0^2,\\
  \label{30}
  &&\qquad-\Delta\Psi_0 =\O_0.
   \eq

We first consider $U_0$. First assume $\gamma \not =2$. Suppose $\sup U_0>0$. Since $U_0(Y)=o(1)$ as $|Y|\to \infty$, the maximum of $U_0$ is attained at some point $Y_0$.  If $Y_0$ is in the interior, then
\eqref{28} implies $U_0(Y_0)=0$, a contradiction to  $\sup U_0>0$. Thus $Y_0$ lies on the $Z$-axis. At $Y_0=(0,Z_0)$, we have $\np\Psi_0=(0,\pt_R \Psi_0)$ by assumption \eqref{Psi-BC}, and $\pt_Z U_0=0$ since $Y_0$ is a maximum point. Thus
\bb
\gl Y\cdot \nabla U_0 +\np\Psi_0 \cdot \nabla U_0=(\gl
Z_0 + \pt_R \Psi_0)\pt_Z U_0=0.
\ee
We get $\left(1-\frac{\gl}{2}\right) U_0(Y_0)=0$,  a contradiction to  $\sup U_0>0$. We conclude $\sup U_0=0$. Similarly we can show $\inf U_0=0$. Thus $U_0 \equiv 0$ in the case $\gamma \not =2$.

We now consider the case $\gamma=2$.
Fix a smooth nonincreasing function $\sigma: [0,\infty)\to [0,\infty)$
so that $\sigma(t)=1$ for $0\le t\le 1$ and $\sigma(t)=0$ for $t\ge 2$.
Using $pU_0^{p-1}\sigma_\rho$ as a test function where
 $\sigma_\rho(Y)=\sigma(|Y|/\rho)$ and $\rho>1$, and denoting $\cc_\rho=\cc \cap B_{3\rho}(0)$, we get
\begin{align*}
0&=-\int_{\cc_\rho} \left\{(\gl Y+\nabla^\perp \Psi_0)\cdot \nabla U\right\}pU_0^{p-1}\sigma_\rho dRdZ
\\
&=-\int_{\cc_\rho} \sigma_\rho(\gl Y+\nabla^\perp \Psi_0)\cdot \nabla U_0^{p}dRdZ
\\
&= \int_{\cc_\rho}  U_0^{p} \nabla  \cdot \left\{\sigma_\rho(\gl Y+\nabla^\perp \Psi_0)\right\}dRdZ -
\int_{\partial\cc_\rho}U_0^{p} \sigma_\rho(\gl Y+\nabla^\perp \Psi_0)\cdot \nu dRdZ.
\end{align*}
Note $\partial\cc_\rho = (\cc \cap \partial B_{3\rho})\cup (\partial  \cc \cap B_{3\rho})$. We have $\sigma_\rho=0$ on $\cc \cap \partial B_{3\rho}$ while
on $ \partial  \cc \cap B_{3\rho}$, $\nu=(1,0)$ and
\bb
(\gl Y+\nabla^\perp \Psi_0)\cdot \nu = \gl R- \partial_Z \Psi_0 = 0
\ee
 by assumption \eqref{Psi-BC} again.
Thus the boundary integral vanishes.
Also note $\nabla  \cdot [\sigma_\rho(\gl Y+\nabla^\perp \Psi_0)]=2\gl \sigma_\rho+
\nabla  \sigma_\rho\cdot (\gl Y+\nabla^\perp \Psi_0)$. We conclude
\begin{align*}
 2\gl  \int_{\cc}  U_0^{p}\sigma_\rho dRdZ
&=- \int_{\cc}  U_0^{p}  \nabla  \sigma_\rho\cdot (\gl Y+\nabla^\perp \Psi_0)dRdZ
\\
&\le C\int_{\rho<|Y|<2\rho} U_0^p (1+\frac {1}{\rho}|\nabla \Psi_0|)dRdZ.
\end{align*}
By assumptions \eqref{UOPsi-infty} and \eqref{integral-assume}, the last integral vanishes as $\rho \to \infty$. We conclude $U_0 \equiv0$.

Now $\O_0$-equation \eqref{29}  is similar to $U_0$-equation \eqref{28} since $U_0=0$.
By the same argument we  get $\O_0\equiv 0$.

By $\Psi_0$-equation \eqref{30}, $\Psi_0$ and $\nabla \Phi_0$ are harmonic. By the boundary conditions \eqref{UOPsi-infty} and \eqref{Psi-BC}, we get $\pt_Z \Psi_0=0$. Thus $\Psi_0=\Psi_0(R)$ is independent of $Z$. By \eqref{30} again, $\Psi_0=aR+b$. By \eqref{UOPsi-infty}, $a=0$. Thus
 $\nabla \Psi_0\equiv 0$.

\medskip

To show that $U_k, \O_k,\nabla \Psi_k=0$ for $k>0$, we prove by induction and assume it has been shown for all smaller $k$. Then  $U_k, \O_k,\Psi_k$ are the leading terms in \eqref{11k}--\eqref{13k} and they satisfy
  \bq\label{19k}
   &&\left(1-\frac{\gl}{2}-k\gl \right) U_k+\gl Y\cdot \nabla U_k =0,\\
   \label{20k}
  &&(1-k\gl)\O_k+\gl Y\cdot \nabla\O_k =0,\\
  \label{21k}
  &&\qquad-\Delta\Psi_k =\O_k.
   \eq
This system is similar to \eqref{28}--\eqref{30}, with the differences being: (i)
 the coefficients of the first terms of \eqref{19k} and \eqref{20k}, due to time derivatives of $(T-t)^{-1+\frac{\gl}{2}+k\gl}$ and $(T-t)^{-1+k\gl}$;
(ii) the nonlinear terms drop off due to  higher powers in $(T-t)^\gl$. Compare
\eqref{3.4}--\eqref{3.6}.

Now the same argument for the case $k=0$ goes through for the case of general $k$. It is in fact easier since $\Psi_k$ does not occur in \eqref{19k} and \eqref{20k}.
The boundary condition \eqref{Psi-BC} is used only once  to show $\pt_Z \Psi_k=0$ in $\cc$.
The decay condition \eqref{integral-assume} is needed only if $1-\frac{\gl}{2}-k\gl =0$ or
$1-k\gl =0$, which does not occur if $k > 1/\gl$. We conclude $U_k=\O_k=\nabla \Psi_k=0$ for all $k\ge 0$.
\end{proof}

\section{Discussion}
\setcounter{equation}{0}
Since the self-similar ansatz of Luo-Hou \cite{hou} is numerically observed, it is robust in some sense. One possible way to explain the discrepancy between \cite{hou} and our results is that the  self-similar singularity is only observed in \cite{hou} in a subregion  of
a time-dependent window  $\mathcal{W}_{\de(t)}$ defined in \eqref{thm2region2}, with 
\bb
\de(t) \le C (T-t)^\gamma.
\ee
In such a case, the self-similar profile $(U,\Omega,\Psi)(R,Z)$ is defined only for $(R,Z)$ in a finite region,
and
the decay condition \eqref{26} is no longer relevant.
Furthermore, even if the self-similar ansatz is valid in the region $\mathcal{D}_{\de,T}$ or in $\mathcal{W}_{\de(t)}$ with $\limsup (T-t)^{-\gamma} \de(t)=\infty$,
the decay condition \eqref{26} makes no distinction between a periodic boundary in the $z$ variable
 or an infinite cylinder, although  it is known that such a difference may change the blow-up behavior. For example, Titi \cite{Titi} reports that the equation
\bb
u_t - u_{xx}+u_x^4=0
\ee
has no blow-up with periodic boundary condition, but has blow-up with Dirichlet boundary condition.

We now explain how an energy consideration suggests \eqref{26} for small $\gamma$.
Suppose the self-similar ansatz \eqref{11}--\eqref{13}
is valid in
the region \eqref{thm2region} for $0<\de<1/2$.
Since the energy of solutions of Euler equations are uniformly bounded in time and $r\sim 1$,
\bb
 \int_{-\delta}^{\delta} \int _{1-\de}^1( |u^\theta|^2 + |u^r|^2 +|u^z|^2 ) drdz<C
\ee
holds uniformly for $t \in (T-\de,T)$.
 By \eqref{11}--\eqref{13} and (\ref{13a}), we get
\bb
\int_{-L}^L\int _{-L}^0 \left\{  (T-t)^{-2+3\gl}|U|^2 + (T-t)^{-2+4\gl} |\nabla \Psi|^2\right\} dRdZ<C
\ee
where $L=\de (T-t)^{-\gamma}  \in (\de^{1-\gamma},\infty)$.
In other words, we have
\bb
\frac 1{L^2}\int _{-L}^0 \int_{-L}^L  |U|^2  dRdZ <C L^{1-\frac 2\gl},\quad
\frac 1{L^2}\int _{-L}^0 \int_{-L}^L  |\nabla \Psi|^2  dRdZ <C L^{2-\frac 2\gl}
\ee
for all large $L$.
This suggests that, in average sense,
\bb
|U(Y)|\le C |Y|^{\frac12-\frac 1\gl},\quad  |\nabla \Psi(Y)|\le C |Y|^{1-\frac 1\gl}.
\ee
It implies $|\nabla \Psi(Y)|=o(|Y|)$ for all $\gl>0$, and $|U(Y)|=o(1)$ for $\gl<2$, as $|Y|\to \infty$. However, this consideration gives no information on the decay of $\O$.
Also note that, the blow-up rate observed in
\cite{hou} is
\bb
\gl \approx 2.91,
\ee
(see \cite[Table 4.9.1]{hou}, where $\gamma$ is denoted as $\hat \gamma_l$),
which is greater than 2, and hence the above consideration does not apply.

 $$ \mbox{\bf Acknowledgements } $$
This research was initiated when the authors visited Tsinghua Sanya International
Mathematics Forum (TSIMF) in December 2013. We thank the referees for very helpful comments.
Chae's research is supported partially by NRF
  Grants no.~2006-0093854 and  no.~2009-0083521. Tsai's research is supported partially by
 NSERC grant 261356-13.

\end{document}